%% file: oo_mssm2.tex
\documentclass[reqno,12pt]{amsart} 

\usepackage[ruled]{algorithm2e}
\renewcommand{\algorithmcfname}{ALGORITHM}
\SetAlFnt{\small}
\SetAlCapFnt{\small}
\SetAlCapNameFnt{\small}
\SetAlCapHSkip{0pt}
\IncMargin{-\parindent}

\usepackage{graphicx}
\usepackage{hyperref} %
\usepackage{amssymb,amsmath,amstext}
\usepackage{latexsym}
\usepackage{microtype} 
\usepackage{algorithmic}
\usepackage{xcolor}

\newtheorem{theorem}{Theorem}
\newtheorem{lemma}{Lemma}
\newtheorem{corollary}{Corollary}

\begin{document}

\markboth{J. Erway and M. Rezapour}{SC-SR1: MATLAB Software for Solving Shape-Changing L-SR1 Trust-Region Subproblems}

\title[A New Multipoint  Symmetric  Secant Method with a Dense Initial Matrix]
      {A New Multipoint  Symmetric  Secant Method with a Dense Initial Matrix}

\author[J. Erway]{Jennifer B. Erway}
\email{erwayjb@wfu.edu}
\address{Department of Mathematics, Wake Forest University, Winston-Salem, NC 2\
7109}

\author[M. Rezapour]{Mostafa Rezapour}
\email{rezapom@wfu.edu}
\address{Department of Mathematics, Wake Forest University, Winston-Salem, NC 2\
7109}

\thanks{J.~B. Erway is supported in part by National Science Foundation grant
IIS-1741264.}

\begin{abstract}
In large-scale optimization, when either forming or storing Hessian
matrices are prohibitively expensive, quasi-Newton methods are often
used in lieu of Newton's method because they only require first-order
information to approximate the true Hessian.  Multipoint symmetric
secant ({\small MSS}) methods can be thought of as generalizations of
quasi-Newton methods in that they attempt to impose additional
requirements on their approximation of the Hessian.  Given an initial
Hessian approximation, {\small MSS} methods generate a sequence of
possibly-indefinite matrices using rank-2 updates to solve nonconvex 
unconstrained optimization problems.  For practical reasons, up to now, the
initialization has been a constant multiple of the identity matrix.
In this paper, we propose a new limited-memory {\small MSS} method for
large-scale nonconvex optimization that
allows for dense initializations.  Numerical results on the
{\small CUTE}st test problems suggest that the {\small MSS} method using
a dense initialization outperforms the standard initialization.
Numerical results also suggest that this approach is competitive with
both a basic L-SR1 trust-region method and an L-PSB method.
\end{abstract}


\maketitle
\makeatletter
\newcommand{\defined}{\mathop{\,{\scriptstyle\stackrel{\triangle}{=}}}\,}

\newcommand{\rfm}[1]{\textcolor{black}{#1}}
\newcommand{\obs}[1]{\textcolor{black}{#1}}
\newcommand{\je}[1]{\textcolor{black}{#1}}
\newcommand{\jeo}[1]{\textcolor{black}{#1}}
\newcommand{\jb}[1]{\textcolor{black}{#1}}
\newcommand{\jbo}[1]{\textcolor{black}{#1}}
\renewcommand{\algorithmcfname}{ALGORITHM}

\newcommand{\bp}{\mathbf{p}}
\newcommand{\bv}{\mathbf{v}}
\newcommand{\bfm}[1]{\mathbf{#1}} 
\newcommand{\bk}[1]{\mathbf{#1}_k} 
\newcommand{\bko}[1]{\mathbf{#1}_{k+1}} 
\newcommand{\bsk}[1]{\boldsymbol{#1}_k} 
\newcommand{\bsko}[1]{\boldsymbol{#1}_{k+1}} 
\newcommand{\bs}[1]{\boldsymbol{#1}} 

\newcounter{pseudocode}[section]
\def\thepseudocode{\thesection.\arabic{pseudocode}}
\newenvironment{pseudocode}[2]%
        {%
        \refstepcounter{pseudocode}%
          \AlgBegin %
               {{\bfseries Algorithm \thepseudocode.}\rule[-1.25pt]{0pt}{10pt}#1}%
        #2}%
           {\AlgEnd}

\newcounter{Pseudocode}[section]
\def\thePseudocode{\thesection.\arabic{Pseudocode}}
\newenvironment{Pseudocode}[2]%
        {%
        \refstepcounter{Pseudocode}%
          \AlgBegin %
               {{\bfseries #1.}\rule[-1.25pt]{0pt}{10pt}}%
        #2}%
           {\AlgEnd}

\def\tnu{\tilde{\nu}}

\newcounter{procedureC}
\newcounter{algorithm saved}
\newenvironment{procedureAlg}[1][H]{%
	\setcounter{algorithm saved}{\value{algocf}} 
    	\setcounter{algocf}{\value{procedureC}}
	\renewcommand{\algorithmcfname}{PROCEDURE}
   \begin{algorithm}[#1]%
  }{\end{algorithm}
  \setcounter{procedureC}{\value{algocf}}
  \setcounter{algocf}{\value{algorithm saved}} 
  }

\makeatother

\section*{Dedication}
\emph{This paper is dedicated to Oleg Burdakov who passed away June 1, 2021.}
\input{intro}

\input{compact}

\input{dense}

\input{oo_implementation}

\input{numerical}

\input{conclusion}
\input{acknowledgements}
\section*{Funding} J.~B. Erway’s research work was funded by NSF Grant IIS-1741264.

\newpage
\bibliographystyle{unsrt}
\bibliography{references}

\end{document}

%% file: intro.tex
\section{Introduction} 
We consider nonconvex large-scale unconstrained optimization problems of the form
\begin{equation}\label{eqn-min}
\min f(x),\end{equation} where $f:\Re^n\rightarrow \Re$ is a twice-continuously
differentiable function.  In this setting, second-order methods are
undesirable when the Hessian is too computationally expensive to
compute or store.  In this paper, we consider the first-order
multipoint symmetric secant 
({\small MSS}) method~\cite{burdakov1983methods} and its limited-memory
extension~\cite{burdakov2002limited}. We propose a new formulation of
the limited-memory {\small MSS} method that allows
for a dense initialization in place of the usual scalar multiple of
the identity matrix.  The dense initialization
does not significantly increase storage costs; in fact,
the dense initialization only requires storing one extra scalar.

\bigskip

Multipoint symmetric secant methods can be thought of as generalizations
of quasi-Newton methods. Quasi-Newton methods generate a sequence of
possibly-indefinite matrices $\{B_k\}$ to approximate the true Hessian of $f$. These methods build
Hessian approximations using
\emph{quasi-Newton pairs} $\{s_k,y_k\}$ that satisfy the \emph{secant
condition} $B_{k+1}s_k=y_k$ for all $k$.  In contrast, {\small MSS} methods seek to impose the secant condition on
\emph{all} the stored pairs: $B_kS_k=Y_k$ for all $k$, where
$$
S_k=\left[ s_{k-1}\, s_{k-2} \, \ldots \,s_0\right] \quad \text{and}
\quad Y_k=\left[ y_{k-1} \,\, y_{k-2} \,\, \ldots y_0
  \right].
$$
It turns out that imposing this requirement
while insisting $B_k$ is a symmetric matrix is, generally speaking, not
possible.  The need to relax these conditions  becomes apparent when
one considers that if $B_kS_k=Y_k$ then it follows that
$S_k^TB_kS_k=S_k^TY_k$, giving that if $B_k$ is symmetric then $S_k^TY_k$
must be symmetric~\cite{schnabel1983quasi}. Unfortunately, $S_k^TY_k$ is usually not symmetric.
Instead of strictly imposing both the multiple secant conditions and symmetry, some of these conditions must be relaxed.

One way to relax these conditions is to ensure the
multiple secant conditions hold but not enforce symmetry.
In~\cite{schnabel1983quasi}, Schnabel generalizes the
Powell-Symmetric-Broyden ({\small PSB}), Broyden-Fletcher-Goldfarb-Shanno
({\small BFGS}), and the Davidon-Fletcher-Powell ({\small DFP})
updates to satisfy $B_kS_k=Y_k$.   However, in  each case,
these generalizations may not  be symmetric and there are conditions
that must be met for the updates to be well-defined.  Alternatively,
one  may relax the multiple secant equations but enforce symmetry.
An  example of this is found in the same
paper when Schnabel proposes perturbing $Y_k$ to obtain a symmetric matrix.

Another type of relaxation takes the form of an approximation of the
true $S_k^TY_k$ using a symmetrization of this matrix.
In~\cite{burdakov1983methods}, the following symmetrization is
proposed:
$$\text{sym}(A)=\left\{
\begin{array}{l}
  A_{ij}, \,\, i\ge j \\
   A_{ji}, \,\, i<j. 
\end{array} \right. $$
With this symmetrization, a {\small MSS} method (and its
limited-memory version described in~\cite{burdakov2002limited}) is a symmetric method,
generating a sequence of symmetric Hessian approximations. (For more details on this symmetrization, see~\cite{burdakov1983methods,burdakov2002limited}.)
A crucial
aspect of these {\small MSS} methods  is that they generate
a sequence of matrices that are not necessarily positive definite.  It
is for this reason that these methods are incorporated into
trust-region methods.  The update formula found
in~\cite{burdakov1983methods,burdakov2002limited} is a rank-2
update formula.  As with traditional quasi-Newton methods, in order to
maintain the low-memory advantages of limited-memory versions of the
methods, a scalar multiple of the identity matrix is used to initialize
the sequence of matrices. (For the duration of the paper, this initialization will be referred to as a ``one-parameter
initialization,'' with the parameter being chosen as a scalar.)  It is
  worth noting that it is possible to use general diagonal matrices as an initialization
  and maintain the low-memory advantage; however, constant diagonal initializations are
  the most popular approach due to their ease of use.
\medskip

In this paper, we propose using a \emph{compact formulation} for the
sequence
of matrices produced by the {\small MSS} method found in~\cite{burdakov1983methods,burdakov2002limited}
to allow the use of two parameters to 
define the initial approximate Hessian.  This initialization exploits
the form of the compact formulation to create a dense initialization
when the two parameters are different.  (It is worth noting that when
the two parameters are chosen to be the same, the initialization simplifies
to the traditional one-parameter initialization.)  
This work is inspired by
the positive results found using a two-parameter initialization for the 
{\small BFGS}  update~\cite{denseInit}.  We demonstrate that the
eigenspace associated with all the eigenvalues is unchanged using two
parameters in lieu of the usual single parameter.  Additionally, some
motivation for good choices of initial parameters are presented.

\medskip

This paper is organized in five sections.  In Section 2, we review
the compact formulation for the limited-memory {\small MSS} matrix.  In Section
3, we derive the proposed limited-memory {\small MSS} method.  This proposed
method is incorporated into a traditional trust-region method in
Section 4.  Numerical comparisons with the {\small MSS} method using
traditional initializations on large unconstrained problems from the
{\small CUTE}st test collection~\cite{cutest} are given in Section 5.
Finally, Section 6 includes some concluding remarks and observations.

\subsection{Notation and Glossary.}
Unless explicitly indicated, $\| \cdot \|$ denotes the Euclidean two-norm
or its subordinate matrix norm.  The symbol $e_i$ denotes the
$i^{\text{th}}$ canonical basis vector whose dimension 
depends on context.

\subsection{Dedication.}
The idea to use a dense initialization for the {\small MSS} method is
due to Oleg Burdakov. An initial conversation with him about this work
occurred in early 2021 and several follow-up emails were exchanged.
Unfortunately, Oleg passed away before the first draft of this paper
was written. We dedicate this paper to him.

%% file: compact.tex
\section{Background}
In this section, we review how {\small MSS} matrices have been
constructed, their compact formulation, and a brief overview of how
they can be used in
{\small MSS} methods for unconstrained optimization.

\subsection{MSS matrices}
Given a sequence of iterates $\{x_i\}$, we define 
\begin{equation}\label{eq-SY}
S_k=\left[ s_{k-1}\, s_{k-2} \, \ldots \,s_0\right] \quad \text{and}
\quad Y_k=\left[ y_{k-1} \,\, y_{k-2} \,\, \ldots y_0 \right],
\end{equation}
where $s_i\defined x_{i+1}-x_i$ and $y_i\defined \nabla
f(x_{i+1})-\nabla f(x_i)$ for $i=0, \ldots, k$.  For simplicity we
assume $S_k$ has full column rank.  (The case when it does not have
full column rank is discussed in Section 5.)  The multiple secant
conditions used to define a multipoint symmetric secant method are
\begin{equation}\label{Dec301}
B_kS_k=Y_k.
\end{equation}
As noted in the introduction both symmetry and the multiple secant conditions cannot usually be enforced.  In this work, the multiple secant conditions are relaxed by applying the symmetrization transformation~\cite{burdakov1983methods,burdakov2002limited}:
\begin{align}\label{ON}
\text{sym}(A)=\begin{cases} 
      A_{ij}& \text{ if } i\geq j,\\
      A_{ji}& \text{ if } i< j, \\
   \end{cases}
\end{align}
to $S_k^TY_k$ to guarantee symmetry.
Suppose we write
\begin{equation}\label{eqn-SYdecompose}
  S_k^TY_k=L_k+E_k+T_k,
\end{equation}
where $L_k$ is the strictly lower triangular, $E_k$ is the diagonal,
and $T_k$ is the strictly upper triangular
part of $S_k^TY_k$.  With this notation,
$\text{sym}\left(S_k^TY_k\right)=L_k+E_k+L_k^T$.
It can be shown that the {\small MSS}  matrix $B_k$ satisfies
\begin{equation}\label{eq:b4}
S_k^TB_kS_k=\text{ $\text{sym}$}(S_k^TY_k),
\end{equation}
for all $0<k\le n$ (see~\cite{brust2018large}).

With this relaxation, the recursion formula to generate {\small MSS} matrices is a rank-2 formula
given by
\begin{equation}\label{eqn-rank2}
B_{k+1}=B_k+\frac{(y_k-B_ks_k)c_k^T+c_k(y_k-B_ks_k)^T}{s_k^Tc_k}
-\frac{(y_k-B_ks_k)^Ts_kc_kc_k^T}{(s_k^Tc_k)^2},\end{equation}
where $c_k\in\Re^n$ is any vector such that $c_k^Ts_i=0$ for all $0\le i <k$
and $c_k^Ts_k\ne 0$~\cite{burdakov1983methods}.   Interestingly, if $c$ is replaced by $\alpha c$ for any
$\alpha\ne 0$, $B_{k+1}$ 
in (\ref{eqn-rank2}) does not change~\cite{burdakov1983methods}.  Thus, $B_{k+1}$ is invariant with
respect to the scaling of $c$.  
The recursion relation (\ref{eqn-rank2}), absent the additional
orthogonality requirements on $c_k$, is a well-known
recursion~\cite{dennis77} that can be used to derive the Powell symmetric
Broyden ({\small PSB}) and the Davidon-Fletcher-Powell ({\small DFP})
update.  Specifially, when $c_k=s_k$ or $c_k=y_k$, the updates are
exactly the {\small PSB} and {\small DFP} update, respectively.  Similar to the {\small PSB} update,
the {\small MSS} update can yield indefinite matrices.

\medskip

The choice of the vectors $\{c_k\}$ are not unique, and thus, the sequence
$\{B_k\}$ is not unique.  In~\cite{burdakov2002limited,brust2018large},
the following choice is made for $c_k$ so  that
$\{B_k\}$ satisfies a least-change problem:
$$c_k \defined \left[I-S_{k}\left(S_{k}^TS_{k}\right)^{-1}
  S_{k}^T\right]s_k.$$
For the duration of this paper, we assume this choice for $\{c_k\}$.
For more details, see~\cite{burdakov1983methods,burdakov2002limited}. 
In the following lemma, we show that despite the relaxation, the original
secant condition is still satisfied.

\begin{lemma}\label{lemma-secant}
{\small MSS} matrices satisfy the secant condition $B_{k+1}s_k=y_k$.
\end{lemma}
\begin{proof}{Lemma \ref{lemma-secant}}
  The proof makes use of the rank-2 formula (\ref{eqn-rank2}):
$$\begin{array}{lcl}
B_{k+1}s_k &=& \left[B_k+\frac{(y_k-B_ks_k)c_k^T+c_k(y_k-B_ks_k)^T}{s_k^Tc_k}
  -\frac{(y_k-B_ks_k)^Ts_kc_kc_k^T}{(s_k^Tc_k)^2}\right]s_k\\
&  & B_ks_k+\frac{(y_k-B_ks_k)c_k^Ts_k+c_k(y_k-B_ks_k)^Ts_k}{s_k^Tc_k}
-\frac{(y_k-B_ks_k)^Ts_kc_kc_k^Ts_k}{(s_k^Tc_k)^2}\\
& = & B_ks_k + (y_k-B_ks_k)+\frac{c_k(y_k-B_ks_k)^Ts_k}{s_k^Tc_k}
-\frac{(y_k-B_ks_k)^Ts_kc_k}{(s_k^Tc_k)}\\
& = & y_k.
  \end{array}$$
\end{proof}

\subsection{Compact formulation and the limited-memory case}
Given an initial $B_0$, 
Brust~\cite{brust2018large} derives the compact formulation $B_k=B_0+\Psi_kM_k\Psi_k^T$, where
\begin{equation}\label{eqn-cptB}
  \Psi_k\defined \begin{bmatrix} S_k & (Y_k-B_0S_k)\end{bmatrix} \quad \text{and}
  \quad
  M_k\defined \begin{bmatrix} W(S_k^TB_0S_k-(T_k+E_k+T_k^T))W & W \\ W & 0\end{bmatrix},
\end{equation}
$W\defined (S_k^TS_k)^{-1}$, and $L_k$, $E_k$, and $T_k$ are defined as in (\ref{eqn-SYdecompose}).
It is worth noting that $B_0S_k$ in (\ref{eqn-cptB}) is not difficult to
compute when $B_0$ is the dense matrix proposed in this paper (see Section 5 for details).  

\medskip

Limited-memory {\small MSS} matrices are obtained by storing no more than
$m$ of the most recently-computed quasi-Newton pairs.  As in typical
limited-memory quasi-Newton methods, $m$ is typically less than 10 (e.g.,
Byrd et al.~\cite{Representations} suggest $m\in[3,7]$).  In the limited-memory 
case, the compact formulation given in (\ref{eqn-cptB}) holds but with $S_k$ and $Y_k$ defined to contain at most $m$ vectors where $m\ll n$, i.e.,
\begin{equation}\label{eq-lSY}
S_k=\left[ s_{k-1}\, s_{k-2} \, \ldots \,s_{k-l}\right]\in\Re^{n\times l} \quad \text{and}
\quad Y_k=\left[ y_{k-1} \,\, y_{k-2} \,\, \ldots y_{k-l} \right]\in\Re^{n\times l},
\end{equation}
where $l\le m$.  In the compact representation (\ref{eqn-cptB}), $M_k$ is at most size $2m\times 2m$, making it a small enough matrix to perform linear algebra operations that are ill-suited for larger (e.g., $n\times n$) matrices.
For the duration of the paper, we assume the {\small MSS} method is
the  limited-memory version where $S_k$ and $Y_k$ are defined by (\ref{eq-lSY}).
\subsection{MSS methods}

The {\small MSS} matrices used in this work have been proposed to solve systems of 
equations~\cite{burdakov1983methods}, bound constrained optimization
problems~\cite{burdakov2002limited}, and unconstrained
optimization~\cite{brust2018large} problems.  In this section, we review how
they can be used to solve (\ref{eqn-min}).
Because {\small MSS} matrices can be indefinite, it is natural to use
this sequence of matrices as approximations to the true Hessian in a
trust-region method for solving nonconvex optimization problems.
We briefly review trust-region methods here, but further details
can be found in standard optimization textbooks (see, e.g.,~\cite{nocedal2006numerical}).

\medskip

At each iterate, trust-region methods use a quadratic approximation $m_k$
to model $f$ at the current iterate in a small region neighborhood about the current iterate $x_k$:
\begin{equation}\label{model}
m_k (x_k + s ) = f(x_k) +g^T_k s +\dfrac{1}{2} s^T B_ks,
\end{equation}
where $g_k=\nabla f(x_k)$ and $B_k\approx \nabla^2 f(x_k)$.  At each iteration, trust-region methods solve the so-called \emph{trust-region subproblem} to obtain a new search direction $s_k\in\Re^n$:
\begin{equation}\label{subproblem}
  s_k=\arg\min\limits_{\| s \| \leq   \Delta_k} g^T_k s +\dfrac{1}{2} s^T B_ks,
\end{equation}
where $\Delta_k$ is the trust-region radius, and $\|\cdot\|$ is any vector norm.
Notice that the difference between the quadratic function in (\ref{subproblem})
  and the model function (\ref{model}) is a constant, which does not affect
  the solution of the subproblem.
Having
obtained a search direction $s_k$, a basic trust-region algorithm computes
a new iterate $x_{k+1}=x_k+s_k$, provided $s_k$ satisfies a sufficient decrease criteria.
(For more details, see, e.g., \cite{nocedal2006numerical,conn2000trust}.)

One of the
advantages of using the Euclidean norm to define the trust region
is that the subproblem is guaranteed
to have an optimal solution. 
The following
theorem~\cite{Gay81,more1983computing} classifies a global
solution to the trust-region subproblem defined using the Euclidean norm:
\begin{theorem}\label{thrm-gay}
Let $\Delta$ be a positive constant.  A vector $s^*$ is a global
solution of the trust-region subproblem (\ref{subproblem}) where
$\|\cdot \|$ is the Euclidean norm if and only if $\|s^*\|\le \Delta$
and there exists a unique $\sigma^*\ge 0$ such that $B+\sigma^* I$ is
positive semidefinite and
\begin{equation}\label{eqn-opt}
  (B+\sigma^*I)s^*=-g \quad \text{and} \quad
\sigma^*(\Delta-\|s^*\|_2)=0.\end{equation}
Moreover, if $B+\sigma^*I$ is positive definite, then the global
minimizer is unique.
 \end{theorem}

 It should noted that these optimality conditions also serve as the
 basis for some algorithms to find the solution to subproblems defined
 by the Euclidean norm.  Specifically, these  algorithms find an optimal
 solution
 by searching for a point that satisfies the conditions 
given in (\ref{eqn-opt}) (see,
e.g.,~\cite{more1983computing,MSS1,burdakov2017efficiently,brust2017solving}.)

The {\small MSS} method proposed in this paper uses a trust-region
method defined by the Euclidean norm.  The method solves each trust-region
subproblem using the ideas in~\cite{brust2017solving} for solving
indefinite subproblems.  The following section overviews these ideas.

\subsection{The  Orthonormal basis L-SR1  (OBS) method.}\label{sec-obs}
The {\small OBS} method~\cite{brust2017solving} is a trust-region method
for solving limited-memory symmetric rank-one ({\small L-SR1}) trust-region
subproblems.  The method is able to solve  subproblems to high  accuracy
by exploiting the  optimality conditions given in Theorem~\ref{thrm-gay}.
{\small L-SR1} matrices are generated by a recursion relation using rank-1
updates that  satisfy the secant condition $B_{k+1}s_k=y_k$  for   all
$k$.  Importantly, as with {\small MSS} matrices, {\small L-SR1} matrices
can be indefinite.  (For details on {\small L-SR1}  matrices and methods
see, e.g.,~\cite{nocedal2006numerical}.)

The {\small OBS} method makes use of the partial spectral decomposition
that can be computed efficiently from the compact
formulation~\cite{burdakov2017efficiently,brust2017solving,simaxErway}.
In particular, suppose the compact formulation of an {\small L-SR1}
matrix $B_{k+1}$ is given by $B_{k+1}=B_0+\Psi M\Psi^T$, with $B_0=\gamma I$.
Further, suppose $\Psi=QR$ is the ``thin'' QR factorization of $\Psi$
where $Q\in\Re^{n\times (k+1)}$ and $R\in\Re^{(k+1)\times (k+1)}$ is invertible.
Then,
$$B_{k+1}=\gamma I + QRMR^TQ^T,$$
where $RMR^T$ is a  small $(k+1)\times (k+1)$ matrix (i.e., the number
of stored quasi-Newton  pairs is $k+1$.)   Given the spectral decomposition
$U\hat{\Lambda} U^T$ of $RMR^T$,  the spectral decomposition of $B_{k+1}$
can be written as $B_{k+1}=P\Lambda P^T$, where
$$\Lambda \defined \begin{bmatrix}\hat{\Lambda}+\gamma I & 0 \\ 0 & \gamma I\end{bmatrix} \quad \text{and} \quad P\defined\begin{bmatrix} QU & (QU)^\perp \end{bmatrix}.$$
  For more details on this  derivation for the {\small L-SR1} case, see~\cite{brust2017solving}.

  The  {\small OBS} method uses the above derivation to determine the eigenvalues of  $B_k$ at  each iteration.  Then, depending on the definiteness of $B_k$, the method computes a  high-accuracy solution using the optimality conditions.
  In some cases, a global solution can be computed directly by formula;
  in other cases, Newton's method  must be used to compute $\sigma^*$
  and then $p^*$ is computed directly by solving $(B+\sigma I)p^*=-g$.
  It is worth noting that in the latter case, Newton's method can be initialized
  to guarantee that Newton converges monotonically without needing any
  safeguarding.  Moreover, a high-accuracy solution can also be computed
  in the  so-called \emph{hard  case}~\cite{brust2017solving}.

  The {\small OBS} method can be generalized to solve trust-region subproblems
  where the approximation to the Hessian admits any compact
  representation.  The  addition of a second parameter to the compact
  formulation (e.g., the proposed dense initialization) requires significant modifications to the code.
  However, the resulting code makes use of the same strategy as the
  {\small OBS} method.   The partial spectral decomposition of a MSS
  matrix is carefully presented in the following section.

%% file: dense.tex
\section{The Dense Initialization}
In this section, we propose a {\small MSS} method with a  dense
initialization.  In order to introduce the dense  initialization,
we begin by demonstrating how to compute a partial spectral decomposition
of an {\small MSS} matrix in order to lay the groundwork for the dense
initialization.

\subsection{The spectral decomposition}\label{sec-eig}
Recall the compact formulation of a {\small MSS} matrix:
\begin{equation}\label{eq-cptMSS}
B_k=B_0+\Psi_kM_k\Psi_k^T,
\end{equation}
where 
\begin{equation}\label{eqn-cptB2}
  \Psi_k\defined \begin{bmatrix} S_k & (Y_k-B_0S_k)\end{bmatrix} \quad \text{and}
  \quad
  M_k\defined \begin{bmatrix} W(S_k^TB_0S_k-(T_k+E_k+T_k^T))W & W \\ W & 0\end{bmatrix},
\end{equation}
$W \defined (S_k^TS_k)^{-1}$, and $L_k$, $E_k$, and $T_k$ are defined as in (\ref{eqn-SYdecompose}).
Note that this factorization does not require $B_0$ to be a scalar multiple of the identity matrix.
We assume that $\Psi_k\in \Re^{n\times 2l}$  and $M_k\in\Re^{2l\times 2l}$ where
$l\le m$, and $m$ is the maximum number of stored quasi-Newton pairs.

For this derivation of the partial spectral decomposition of $B_k$, we 
assume $B_0=\gamma_k I$ for simplicity.
If the ``thin'' {\small QR} factorization of $\Psi_k$ is $\Psi_k=QR$, where $Q\in\Re^{n\times 2l}$
and $R\in\Re^{2l\times 2l}$ then
$$B_k=B_0+QRM_kR^TQ^T.$$
The matrix $RM_kR^T\in \Re^{2l\times 2l}$ is small, and thus, its spectral decomposition can
be computed in practice.  Let $U\hat{\Lambda}U^T$ be the spectral
decomposition of $RM_kR^T\in \Re^{2l\times 2l}$.  This gives us
\begin{equation}\label{eqn-QU}
B_k  =  B_0+QU\hat{\Lambda}U^TQ^T,
\end{equation}
  and so,
  $B_k=P\Lambda P^T$
  where 
  \begin{equation}\label{eqn-spectral2}
  \Lambda \defined \begin{bmatrix} \hat{\Lambda}+\gamma_k I & 0
    \\ 0 & \gamma_k I\end{bmatrix} \quad \text{and} \quad P\defined
  \begin{bmatrix} QU & (QU)^\perp\end{bmatrix}.
\end{equation}
Here, $\Lambda$ is a diagonal matrix whose first $2l$ entries are
$\hat{\lambda}_i+\gamma_k$, where $\hat{\Lambda}=\text{diag}(\hat{\lambda}_1,
  \ldots, \hat{\lambda}_{2l})$, and the (2,2)-block of $\Lambda$ contains
  the eigenvalue $\gamma_k$ with multiplicity  $n-2l$.  The first $2l$ columns of $P$
  are made up of $QU$, which is orthogonal since it is the product of two orthogonal
  matrices.  In contrast, $(QU)^\perp\in\Re^{n\times (n-2l)}$ is very computationally
  expensive to compute and will never be explicitly used.
  Finally, provided $\Psi_k$ has full column rank, it is worth noting
  that it is possible to form $QU$ without storing $Q$; namely,
  $QU$ can be computed using the formula $QU=\Psi_kR^{-1}U$.
  For simplicity, we define $P_\parallel = QU$ so that $P=\begin{bmatrix} P_\parallel &
    P_\perp\end{bmatrix}$, where $P_\perp=P_\parallel^\perp$.
  Notice that this the same form as the spectral decomposition in (\ref{eqn-spectral2})
  but with different-sized matrices.

  \subsection{The dense initial matrix}\label{subsec-dense}
  In the previous subsection, we assumed that $B_0=\gamma_k I$, i.e., a
  scalar multiple of the identity matrix.  In this subsection, we show
  how the spectral decomposition of $B_k$ can be used to design a
  two-parameter dense initialization.

  \medskip

  Consider the spectral decomposition $B_k=P\Lambda P^T$ where
  $\Lambda$ and $P$ are defined in (\ref{eqn-spectral2}).  In the case
  that $B_0=\gamma_k I$, then $B_0=\gamma_k PP^T$ since $P$ is
  orthogonal.   Moreover, $B_0$ can be written as 
  $$B_0 = \gamma_k P_\parallel P_\parallel^T + \gamma_k P_\perp P_\perp^T.$$
  The dense initialization is formed by assigning a different parameter
  to each subspace, i.e.,
  \begin{equation}\label{eqn-B0}
    \tilde{B}_0 = \zeta_k  P_\parallel P_\parallel^T + \zeta^C_k P_\perp P_\perp^T,
  \end{equation}
  where $\zeta_k$ and $\zeta^C_k$ can be updated each iteration.

  \begin{theorem}\label{theorem1}
Given $\tilde{B}_0$ as in (\ref{eqn-B0}), then $B_k$ in
(\ref{eq-cptMSS}) 
has the following spectral decomposition:
\begin{equation}\label{eqn-spectral3}
  B_k=\begin{bmatrix} P_\parallel & P_\perp \end{bmatrix}
  \begin{bmatrix} \hat{\Lambda}+\zeta_k I & 0 \\ 0 &
    \zeta_k^C I\end{bmatrix}
  \begin{bmatrix} P_\parallel^T \\ P_\perp^T \end{bmatrix}.
\end{equation}
\end{theorem}

\begin{proof}
By (\ref{eqn-QU}) and (\ref{eqn-B0}), we have
\begin{align*}
B_k&=\tilde{B}_0+\Psi_kM_k\Psi_k^T\\
&=\tilde{B_0}+P_\parallel \hat{\Lambda} P_\parallel^T\\
&= \zeta_k  P_\parallel P_\parallel^T + \zeta^C_k P_\perp P_\perp^T+ P_\parallel \hat{\Lambda} P_\parallel^T\\
   &=  \begin{bmatrix} P_\parallel & P_\perp \end{bmatrix}
  \begin{bmatrix} \hat{\Lambda}+\zeta_k I & 0 \\ 0 &
    \zeta_k^C I\end{bmatrix}
  \begin{bmatrix} P_\parallel^T \\ P_\perp^T \end{bmatrix}.
\end{align*}
\end{proof}

 \begin{corollary}\label{corollary3}
   The eigenspace associated with the eigenvalues of $B_k$ is not
   changed
   using the dense initialization (\ref{eqn-B0}) in lieu of the one-parameter
   initialization.
 \end{corollary}

 \subsection{Parameter choices} \label{parameters}

One possible
motivation to pick $\zeta_k$ is to better ensure that the multiple
secant conditions hold, i.e., $B_kS_k=Y_k$ for all $k$.  Thus,
$\zeta_k$ can be selected to be the solution of the following
minimization problems:

\begin{equation}\label{eqjan28-3}
  \tilde{B}_0=\emph{arg}\min_{B_0}\left \| B_0S_{k}-Y_{k}  \right\|_F^2
  \quad \text{or} \quad   \tilde{B}_0=\emph{arg} \min_{B_0}\left\|
  B_0^{-1}Y_{k}-S_{k} \right\|_F^2,
\end{equation} 
where $\|.\|_{F}$ denotes the Frobenius norm.

\medskip

The following two lemmas and their corollaries are useful for solving both
minimization problems in (\ref{eqjan28-3}).

\begin{lemma}\label{lemma-S}
The following relationships hold between $S_k$, $\Psi_k$ and
$P_\parallel$:  (i) $\text{Range}(S_k)\subseteq \text{Range}(\Psi_k)$, (ii) $\text{Range}(S_k)\subseteq
\text{Range}(P_\parallel)$, and (iii) $P_\parallel P_\parallel^T S_k = S_k$.
\end{lemma}
\begin{proof}
  Since $\Psi_k = \begin{bmatrix} S_k & Y_k-B_0S_k\end{bmatrix}$ it
  follows immediately that $\text{Range}(S_k)\subseteq \text{Range}(\Psi_k)$.
  Note that $P_\parallel =\Psi_k  R^{-1}U$, it follows that
  $\text{Range}(\Psi_k)=\text{Range}(P_\parallel)$.
  Finally, $P_\parallel P_\parallel^T S_k=S_k$ since $\text{Range}(S_k)\subseteq
  \text{Range}(P_\parallel)$.
\end{proof}

\begin{corollary}\label{cor-S}
     Suppose $\tilde{B}_0$ is as in (\ref{eqn-B0}).  Then, $\tilde{B}_0S_k=\zeta_k S_k$.
\end{corollary}
\begin{proof}
    By Lemma~\ref{lemma-S},
 $$
    \tilde{B}_0S_k = \left(\zeta_k  P_\parallel P_\parallel^T + \zeta^C_k P_\perp P_\perp^T\right)S_k\\
     =   \zeta_k  P_\parallel P_\parallel^TS_k\\
   =  \zeta_k S_k.
  $$
\end{proof}

\begin{lemma}\label{lemma-Y}
The following relationships hold between $Y_k$, $\Psi_k$ and
$P_\parallel$:  (i) $\text{Range}(Y_k)\subseteq \text{Range}(\Psi_k)$, (ii) $\text{Range}(Y_k)\subseteq
\text{Range}(P_\parallel)$, and (iii) $P_\parallel P_\parallel^T Y_k = Y_k$.
\end{lemma}
\begin{proof}
  To show (i), we prove that $Y_k$ can be written as a linear combination of the
  columns of $\Psi_k$.  Namely, by Corollary \ref{cor-S},
  $$
  Y_k=\zeta_kS_k+\left(Y_k-\tilde{B}_0S_k\right)=\begin{bmatrix}S_k & Y_k-\tilde{B}_0S_k\end{bmatrix}
  \begin{bmatrix} \zeta_kI \\ I\end{bmatrix} = \Psi_k\begin{bmatrix} \zeta_kI\\I
    \end{bmatrix}.
  $$
    Since $P_\parallel =\Psi_k  R^{-1}U$, it follows that
    $\text{Range}(Y_k)\subseteq\text{Range}(P_\parallel)$, giving (ii).
  Finally, $P_\parallel P_\parallel^T Y_k=Y_k$ since $\text{Range}(Y_k)\subseteq
  \text{Range}(P_\parallel)$.
  \end{proof}

\begin{corollary}\label{cor-Y}
     Suppose $\tilde{B}_0$ is as in (\ref{eqn-B0}).  Then, $\tilde{B}_0^{-1}S_k=(1/\zeta_k) S_k$.
\end{corollary}
\begin{proof}
    By Lemma~\ref{lemma-Y},
$$    \tilde{B}_0^{-1}Y_k = \left(\frac{1}{\zeta_k}  P_\parallel P_\parallel^T +
    \frac{1}{\zeta^C_k} P_\perp P_\perp^T\right)Y_k\\
   =   \frac{1}{\zeta_k}  P_\parallel P_\parallel^TY_k\\
   =  \frac{1}{\zeta_k} Y_k.
$$
\end{proof}

\medskip

The optimal values for $\zeta_k$ that solve (\ref{eqjan28-3}) are
given in the following theorem.

\begin{theorem}\label{theoremjan2888-2}
 If $\tilde{B}_0$ is chosen to be in the form of (\ref{eqn-B0}),
  the solutions to
  $$  \hat{B}_0\defined\emph{arg}\min_{
    \tilde{B}_0}\left \| \tilde{B}_0S_{k}-Y_{k}  \right\|_F^2
 \quad \text{and} \quad   \bar{B}_0\defined\emph{arg} \min_{\tilde{B}_0}\left\|
  \tilde{B}_0^{-1}Y_{k}-S_{k} \right\|_F^2
  $$
  are, respectively, $\hat{B}_0=\hat{\zeta}_k P_\parallel P_\parallel^T + \hat{\zeta}_k^C P_\perp
    P_\perp^T$ and
     $\bar{B}_0=\bar{\zeta}_k P_\parallel P_\parallel^T + \bar{\zeta}_k^C P_\perp
    P_\perp^T,$   where
\begin{equation}\label{eqn-oneparam}
  \hat{\zeta}_k=\frac{\text{trace}(S_k^TY_k)}{\text{trace}(S_k^TS_k)}\quad \text{and}\quad
  \bar{\zeta}_k=\frac{\text{trace}(Y_k^TY_k)}{\text{trace}(S_k^TY_k)},
\end{equation}
  and $\hat{\zeta}_k,\bar{\zeta}_k\in\Re$.
\end{theorem}
\begin{proof}
Consider the problem $$  \hat{B}_0\defined\emph{arg}\min_{
  \tilde{B}_0}\left \| \tilde{B}_0S_{k}-Y_{k}  \right\|_F^2,$$
which can be restated as finding $\hat{\zeta}_k$ and $\hat{\zeta}_k^C$
that solve the following optimization problem:
\begin{equation}\label{eqn-zzc}
  \emph{arg}\min_{\hat{\zeta}_k,\hat{\zeta}_k^C	}
\left\| \left(\hat{\zeta}_k P_\parallel P_\parallel ^T+\hat{\zeta}_k^{C} P_\perp P_\perp ^T\right)S_{k}-Y_{k}  \right\|^2_F.
\end{equation}
By Lemma~\ref{lemma-S}, the problem (\ref{eqn-zzc}) is equivalent to
$$
\emph{arg}\min_{\hat{\zeta}_k}
\left\| \hat{\zeta}_k S_k -Y_{k} \right\|^2_F,
$$
which can be written as follows:
\begin{equation}\label{eqn-1D}
\emph{arg}\min_{\hat{\zeta}_k}\left\{
\text{trace}\left( \left(\hat{\zeta}_k S_k-Y_k\right)^T \left(\hat{\zeta}_k S_k-Y_k\right)\right)\right\}.
\end{equation}
The minimization problem (\ref{eqn-1D}) is a one-dimensional optimization problem in
$\hat{\zeta}_k$ that can be solved by differentiating the objective function in (\ref{eqn-1D}) with
respect to $\hat{\zeta}_k$ and setting the result to zero:
$$
  \frac{d}{d\hat{\zeta_k}} \left[ \text{trace}(\hat{\zeta}_k^2S_k^TS_k -
  2\hat{\zeta}_kS_k^TY_k + Y_k^TY_k) \right] =
   2\hat{\zeta}_k\text{trace}(S_k^TS_k)-2\text{trace}(S_k^TY_k ) = 0.
$$
This yields the following
solution:

\begin{equation}
\label{eqn-zetamin}
  \hat{\zeta}_k=\frac{\text{trace}(S_k^TY_k)}{\text{trace}(S_k^TS_k)}.
\end{equation}
Note that $\hat{\zeta}_k$ in (\ref{eqn-zetamin}) is the minimizer
since the function to be minimized is both convex and quadratic
in $\hat{\zeta}_k$.  The solution $\bar{B}_0$ can be obtained using a similar process.
\end{proof}

This strategy for choosing the two parameters does not specify
$\hat{\zeta}_k^C$; in fact, this paramater is irrelevant to solving
the minimization problem.  Also, notice that when $k=1$,
$\bar{\zeta}_k$ is the conventional initialization for the {\small
  BFGS} method and sometimes referred to as the Barzilai and Borwein
({\small BB}) initialization~\cite{bb88}.

 \medskip

 It is also worth noting that applying either choice of $\zeta_k$ from this
 theorem could result in a choice of $\zeta_k$ that is negative.  The
 disadvantage of negative $\zeta_k$ is that the initial $B_0$ will
 have $2l$ negative eigenvalues.  While rank-2 updates could shift
 these eigenvalues into positive territory, the subproblem solver will
 need to be computing solutions on the boundary of the trust region as
 long as any eigenvalue of $B_k$ is negative.  This is undesirable for
 two reasons (i) at a local solution of (\ref{eqn-min}) it is
 desirable that $B_k$ is positive definite and (ii) it is more
 computationally expensive for the subproblem solver to compute
 solutions on the boundary than in the interior of the trust region.

 \subsection{Selecting $\zeta_k^C$}
 While the subspace $P_\parallel$ is constructed from the pairs of
 current limited-memory iterates $\{s_i,y_i\}$, $i=k-l,\ldots, k-1$,
 the subspace $P_\perp$ is unknown.  In~\cite{denseInit}, Brust et
 al. propose suggestions for the second parameter in the proposed
 dense initialization for an {\small {L-BFGS}} method.  These
 suggestions are based on the trust-region subproblem solver that is
 applicable in our proposed method. Indeed, the same comments
 apply to the generalized {\small OBS} method found in
 Section~\ref{sec-obs}.  There is in an inverse relationship between
 $\zeta_k^C$ and the subspace $P_\perp$.  Specifically, if $\zeta_k^C$
 is large then the component of the solution to the trust-reigon
 subproblem that lies in the subspace $P_\perp$ will be small in norm.
 See~\cite{denseInit} for details.  However, choices of $\zeta_k^C$
 that are too large or too small may lead to poorly-conditioned
 approximate Hessians.  Strategies such as
 taking $\zeta_k^C$ to be convex combnations of the current $\zeta_k$
 and previous values of $\zeta_k$ have been found to be beneficial
 in the {\small BFGS} case~\cite{denseInit}.

%% file: oo_implementation.tex
\section{The Proposed MSS Method}
In this section, we present the proposed {\small MSS} method.  This method
uses a trust-region method whose subproblems are defined
using a {\small MSS} matrix.  The result is a type of limited-memory quasi-Newton trust-region
method whose approximate Hessian can be indefinite at any given iteration.  Finally, 
this section
concludes with a discussion of the computational cost
of  the proposed  method.

\subsection{Full rank assumptions}\label{fullrankS}

\subsection{Ensuring $S_k$ has full rank}\label{fullrankS}

In order for the compact formulation to exist, it must be the case
that $S_k$ has full rank so that $W=(S_k^TS_k)^{-1}$ 
in (\ref{eqn-cptB}) is well-defined.
To ensure this, the $LDL^T$ factorization
of $S_k^TS_k$ is computed to find linearly dependent columns in $S_k$.
Specifically, small diagonal elements in  $D$ 
correspond to columns of $S_k$ that are linearly dependent.  Let
$\hat{S}_k$ denote the matrix $S_k$ after removing the linearly
dependent columns. The resulting matrix $\hat{W}\defined (\hat{S}_k^T\hat{S_k})^{-1}$ is
positive definite, and thus, invertible.  This strategy is not new and
was first proposed in this context in~\cite{burdakov2017efficiently}.  For
notational simplicity, we assume that $S_k$ has linearly independent
columns for the remainder of this section because if it is not full
rank, linear dependence will be removed by this procedure.  For the proposed
{\small MSS} method, a full rank $S_k$ is used to define $\Psi_k$.

\subsection{Ensuring $\Psi_k$ has full rank}\label{fullrankPsi}
It
is also desirable that $\Psi_k$ is full rank to avoid any adverse
affects on the condition number of $B_k$.
If $\Psi_k$ is not
full rank, then $\hat{\Lambda}$ will have at least one diagonal entry
that is close to zero (the number of small diagonal elements of $R$
will correspond to the number of small eigenvalues of $RMR^T$).  Without loss
of generality,
assume that $i$ is such that $\hat{\lambda}_i+\zeta_k\approx \zeta_k$,
i.e., $\hat{\lambda}_i\approx 0$.  In the cases when
$|\hat{\lambda}_i+\zeta_k|\approx |\zeta_k|$ is the largest or smallest
approximate eigenvalue of $B_k$ in absolute value, the condition number of $B_k$ will be
negatively affected, possibly detrimentally.  In contrast, if
$|\hat{\lambda}_j+\zeta_k|\le|\zeta_k|\le|\hat{\lambda}_{\hat{k}}+\zeta_k|$
for some $j,\hat{k}\in \{0,1,\ldots, 2\text{rank}(S)\}$ then the condition number
of $B_k$ will not be impacted.  For this reason, ensuring that $\Psi$
has full rank will help prevent unnecessary ill-conditioning.

In order to enforce that $\Psi_k$ has  full rank, we also  perform the  $LDL^T$
factorization of $\Psi_k^T\Psi_k$ to find linearly dependent
columns of $\Psi_k$, similar to the procedure to ensure that  $S_k$ is full
rank.  
\\

\subsection{$LDL^T$ and its effect on the eigendecomposition}
Using the $LDL^T$ factorization of $\Psi_k^T\Psi_k$ in lieu of the $QR$ factorization of $\Psi_k$ in Section~\ref{sec-eig}
requires redefining the  eigendecomposition of $B_k$.
The following presentation is based on~\cite{Oleg2021}.  For notational simplicity, the subscript $k$ is dropped for all matrices.

Suppose $\Pi^T\Psi^T\Psi\Pi=LDL^T$ is the $LDL^T$ factorization of $\Psi^T\Psi$ and $\Pi$ is  a permutation matrix\footnote{For the numerical
results, MATLAB's built-in
\texttt{ldl} command, which uses rook pivoting\cite{matlab,higham02}, is used.}.  Let $J$ be the  set of  indices
such that its corresponding diagonal entry in $D$ and that of it's
``pair'' is sufficiently large (see
Section \ref{sec-complex}).
Then, $$\Psi^T\Psi\approx
\Pi\L_{\dagger}D_{\dagger}L_{\dagger}^T\Pi^T=\Pi R_{\dagger}^TR_{\dagger}\Pi^T,$$
where
$L_\dagger$ is the matrix $L$ having  removed  columns indexed by an element
not in  $J$, $D_\dagger$ is the matrix $D$ having  removed any rows and columns indexed by an element not in  $J$, 
$R_\dagger \defined \sqrt{D_\dagger}L_\dagger^T\in\Re^{r\times {2l}}$, and
  $r=|J|$.  
  This gives the approximate decomposition of $B_k$:
  \begin{equation}\label{eqn-dagger}
  B_k=B_0+Q_\dagger R_\dagger \Pi^TM\Pi R_\dagger^TQ_\dagger^T,
  \end{equation}
  where $Q_\dagger\defined (\Psi\Pi)_\dagger R_\ddagger^{-1}$,
  $(\Psi\Pi)_\dagger$ is the matrix $\Psi\Pi$ having deleted columns
  not indexed in $J$, and $R_\ddagger$ is $R_{\dagger}$ having removed
  columns not indexed in $J$.  (For more details on this derivation,
  see~\cite{Oleg2021}.)

  The partial eigendecomposition proceeds as in Section~\ref{sec-eig}
  but with $R_\dagger \Pi^T M \Pi R_\dagger^T$ in lieu of $RMR^T$.
  Finally, as in~\cite{Oleg2021}, $P_\parallel$ can be computed as
  \begin{equation}\label{eqn-pparallel}P_\parallel=(\Psi\Pi)_\dagger R_\ddagger^{-1}U.\end{equation}


\subsection{The benefits of two $LDL^T$ factorizations}

While ensuring $\Psi_k$ is full rank is sufficient to identify the columns of $S_k$ that 
are linearly independent, there is a disadvantage to computing only the $LDL^T$
factorization of $\Psi_k^T\Psi_k$ in lieu of the $LDL^T$ factorizations of both of $S_k^TS_k$ and
$\Psi_k^T\Psi_k$.  Namely, using the $LDL^T$ factorization of $\Psi_k^T\Psi_k$
to identify linearly independent columns of $S_k$ has the potential of forcing the loss
of more information.  To see this, consider what occurs
when only one $LDL^T$ factorization is performed.  If the factorization
of $\Psi_k^T\Psi_k$ leads to removing a column associated with $S_k$, (e.g., the column $s_i$), then the corresponding
column of $Y_k$, (i.e., $y_i$), must also be removed in order that $S^T_kY_k$ is computable and the dimensions of $M$ are well-defined.  
Similarly, if the column of $y_i-B_0s_i$ of $\Psi_k$
needs to be removed, then to maintain dimension agreement in the compact formulation,
both $s_i$ and $y_i$ from $S_k$ and $Y_k$, respectively, must be removed to compute $M$.  In other words, columns of $\Psi_k$
must be removed in pairs.  This can lead to deleterious effects when, for example, $\Psi_k$
for a matrix of size $n\times 2l$ and
$\hat{l}$ columns are found to be linearly independent corresponding to different subscripts.  In this case, since $l$ is typically chosen to be small, potentially few columns of $\Psi_k$ may
remain; moreover, in the extreme case when $\hat{l}=l$ then
no columns of $\Psi_k$ will be preserved since columns must be removed in pairs unless additional measures are taken. 

In contrast, performing two $LDL^T$ factorizations helps mitigate information loss.  Namely, if $s_i$ is called
upon to be removed due to the $LDL^T$ factorization of $S_k^TS_k$, it is necessary to remove $y_i$ so that  $S_k^TY_k$ can be computed. The remaining columns of $S_k$ and $Y_k$ are used to form $\Psi_k$ (and $M$).  However, if there is further dependency in $\Psi_k$, after having removed dependent columns in $S_k$, then (1) only one column needs to be removed at a time (i.e., if the column $y_i-B_0s_i$ needs to be removed, the column $s_i$ need not be removed from $\Psi_k$)
and (2)
$M$ does not need to be further altered.  To see this, consider equation~(\ref{eqn-dagger}), where columns of $\Psi_k$ are removed, $R_\dagger$ and $R_\ddagger$ are formed, and $M$ remains intact.  In other words, there is no additional information loss in $M$ when columns in $\Psi_k$ are found to be linearly dependent after
having already removed linearly dependent columns in $S_k$ (the first block of $\Psi_k$).  Thus, performing two $LDL^T$ decompositions removes the requirement that (1) columns from $\Psi_k$ are always removed in pairs and (2) mitigates additional information loss in $M$.

\subsection{MSS algorithm with a dense initial matrix}
The following algorithm is the proposed {\small MSS} algorithm.
Details follow the algorithm.  Brackets are used to denote quantities that
are stored using more than one variable.  (For example, $[Bp]=-g$
means to store $-g$ in the variable name $Bp$.)  In the below
algorithm, $m$ denotes the maximum number of stored quasi-Newton pairs
and $l$ denotes the current number of stored pairs.  The {\small MSS}
algorithm requires an initial $x$ as well as a way to compute the
function $f$ and its gradient $g$ at any point.


\begin{procedureAlg}
  \begin{algorithmic}[1]
   \caption{ MSS Method with dense initialization}\label{alg-mss}
     \STATE \textbf{Input:} $x$ 
     \STATE Pick $0\ll \tau<\hat{\tau}$;
     \STATE Set $m\in [3,7]$ with $m\in\mathbb{Z}^+$, $\Delta=1$;
        \STATE Set $f\gets f(x)$ and $g\gets g(x)$;  $l\gets 0$;
        \STATE $p\gets -g$; store $[Bp]=-g$;
        	\STATE Use an Amijo backtracking line search to find $\alpha$;  $x\gets x+\alpha p$; 
   \STATE Update $f\gets f(x)$ and $g\gets g(x)$;
   \STATE Update $s$ and  $y$; $l\gets 1$;
   \STATE Update $\Delta$ using Algorithm (\ref{alg-delta});
     \WHILE{not converged} 
      \IF{$l\le m$,}
      \STATE {Add $s$ to $S$ and add $y$ to $Y$;  $l\gets l+1$;}
      \ELSE \STATE{Remove oldest pair in $S$, $Y$, and corresponding column in $\Psi$;}
      \STATE {Add $s$ to $S$ and add $y$ to $Y$;}
      \ENDIF
      \STATE {Compute $\zeta$ and $\zeta^C$ (e.g., Section \ref{parameters}); Update $\Psi$;}
    \STATE {Ensure $S$ and $\Psi$ are full rank using Section \ref{fullrankS}; Store $R_\dagger$, $R_\ddagger$;}
    \STATE{Update $M$;}
    \STATE {Compute the spectral decomposition $U\hat{\Lambda}U^T$ of $R_{\dagger}\Pi^T M \Pi R_{\dagger}^T$;}
    \STATE {$\Lambda\gets \hat{\Lambda}+\zeta I$;}
   \STATE {{\small$[p,Bp]$=$\text{obs}^*$($g,\Psi,R_\ddagger^{-1},U,\Lambda,\zeta,\zeta^C,\Delta$)};}
   \STATE {$x\gets x+p$;}
  \STATE  {$f\gets f(x)$ and  $g\gets g(x);$}
\STATE  {Update $s$ and $y$;}
\STATE  {Update $\Delta$ using Algorithm (\ref{alg-delta});}
\ENDWHILE
      \STATE  \textbf{Output:} $x$ 
\end{algorithmic}
\end{procedureAlg}

Algorithm \ref{alg-mss} begins by letting $B_0=I$, and thus, the
search direction is the steepest descent direction.  (Future iterations
use the dense initialization.)  The update to  $\Delta$ at each
iteration makes use of Algorithm \ref{alg-delta}, which is a standard
way to update the trust-region radius (see, for example,~\cite{nocedal2006numerical}).
In Algorithm~\ref{alg-delta},  the variable $\rho$ stores the ratio between the actual change
and the predicted change.  If this ratio is sufficiently large,
then the step is accepted and the trust-region radius is possibly updated.
If the ratio is too small, the step is rejected and $\Delta$ is decreased.
Note that the implicit update to $B_k$, i.e., the acceptance of
a new quasi-Newton pair (lines 8 and 25 in  Algorithm~\ref{alg-mss}),
is independent of whether the step is accepted in Algorithm~\ref{alg-delta}.

The $LDL^T$ decomposition is used twice in line 18 in Algorithm~\ref{alg-mss}.  Each time,
the $i$th diagonal entry of $D$ is considered to be sufficiently large provided
$$D_{ii}>\text{tol}*\max_{j}|D_{jj}|,$$
where $D_{jj}$ denotes the $j$th diagonal entry in $D$ and tol is
a small positive number (see Sections \ref{fullrankS}-\ref{fullrankPsi}
for more details).
Finally, in line  (22), a modified \texttt{obs} method (denoted
by the asterisk) is used to solve the trust-region subproblem.

\begin{procedureAlg}
  \begin{algorithmic}[1]
\STATE \textbf{Fix:} $0<\eta_1<\eta_2<1$, $0<\gamma_1, \gamma_p<1<\gamma_2$;
    \STATE $\rho\gets (f_{\text{new}}-f)/(-g^Tp-0.5p^T[Bp],0)$;
    \IF{$\rho\ge\eta_1$,} 
     \STATE {$f\gets f_{\text{new}}$; $x\gets x+p$ and $g\gets \nabla f(x)$;}
     \IF {$\rho\ge\eta_2$,}
     \IF {$\|p\|> \gamma_p \Delta$,}
     \STATE{$\Delta\gets \gamma_2\Delta$;}
     \ENDIF
             \ENDIF
      \ELSE
       \STATE{ $\Delta\gets \gamma_1\Delta$;}
      \ENDIF
      \caption{ Update the trust-region radius}\label{alg-delta}
  \end{algorithmic}
  \end{procedureAlg}

\begin{theorem}\label{theorem-convergence}
  Let  $f:\Re^n\rightarrow \Re$ be a twice-continuously differentiable and
  bounded below. Suppose there exists scalars $c_1,c_2>0$ such that 
 \begin{equation}\label{Jan22-2}
 \|\nabla^2 f(x)\| \leq c_1 \quad \text{and} \quad \|B_k\| \le  c_2, 
 \end{equation}
 for all $x \in \mathbb{R}^n$ and for all $k$. Consider the sequence $\{x_k\}$ generated
 by Algorithm \ref{alg-mss} with  $\zeta, \zeta^C\in\Re$.  
Then,
$$
\lim \limits_{k \to \infty}\|\nabla f(x_k)\|=0.
$$
\end{theorem}

\begin{proof}
  Since the {\small OBS} method generates high-accuracy solutions to the
  trust-region subproblem, all the assumptions of Theorem 6.4.6 in~\cite{conn2000trust} are met, proving convergence.
\end{proof}
 
\subsection{Computational complexity} \label{sec-complex}
In this section, we consider the computational cost per iteration of
Algorithm (\ref{alg-mss}). 
The $LDL^T$ factorization of $S_k^TS_k$ and $\Psi_k^T\Psi_k$
are each $\mathcal{O}(l^3)$.  Moreover, if the $LDL^T$ factorization is
updated each iteration instead of computed from scratch, this cost
drops to $\mathcal{O}(l^2)$.  The spectral decomposition's dominant
cost is at most $\mathcal{O}(l^3)=\left(l^2/n\right)\mathcal{O}(nl)$, depending on the rank of $\Psi_k$.
Finally, the \texttt{obs} method requires two  matrix-vector products
with $P_\parallel$, whose dominant cost is
 $\mathcal{O}(4nl)$ flops to solve each trust-region subproblem.
(For comparison, this is the same as the dominant cost of {\small L-BFGS}~\cite{Nocedal80}.)


%% file: numerical.tex
\section{Numerical Experiments}
The {\small MSS} method was developed to solve large-scale unconstrained nonconvex optimization problems.
For this reason, the experiments in this section test the performance of the {\small MSS} method on a set of large-scale
unconstrained optimization problems from the {\small CUTE}est test set~\cite{cutest}.  Specifically, for these experiments, we considered
all problems from
 {\small CUTE}est test set with $n\ge 1000$ of the classification "OUR2"\footnote{See \url{https://www.cuter.rl.ac.uk/Problems/mastsif.shtml} for further classification information.},
which includes all problems with an objective function that is nonconstant, nonlinear, nonquadratic, and not the sum of squares.
This selection yielded 60 problems.  There were 11 problems on which no method converged and these were removed from the test set resulting in a total of 49 problems on which results are presented.

$$ $$

For these experiments, the trust region was defined using the
Euclidean norm, and Algorithms \ref{alg-mss}-\ref{alg-delta} were run
with the following parameters: $\eta_1=0.01,$ $\eta_2=0.75$,
$\gamma_1=0.5$, $\gamma_2=2$, and $\gamma_p=0.8$.

In our experiments, we found that updating $s$ and $y$ only when
\begin{equation}\label{eqn-sy-positive}s_i^Ty_i>\epsilon\|s_i\|\|y_i\|,
  \end{equation}
where $\epsilon$ is machine precision, slightly outperformed the
strategies of updating $s$ and $y$ regardless of the sign and magnitude of their inner product.
Thus, for all experiments reported in this section, we updated the  pairs $(s,y)$
only when they satisfied (\ref{eqn-sy-positive}).

\medskip

For these experiments, the algorithms terminated successfully if  iterate $x_k$ satisfied
$$\|g(x_k)\|\le\max(\tau_g*\|g_0\|,\tau_g),$$
where $\tau_g=1e^{-5}$ and $g_0=g(x_0)$.  On the other hand, the methods terminated unsuccessfully
if either of the following occurred: the number of iterations reached $2n$, the number of function evaluations
reached $100n$, or $\Delta$ became smaller than $100\epsilon$, where $\epsilon$ denotes machine precision.

The results of experiments are partially summarized using performance profiles, proposed by
Dolan and Mor\'{e}\cite{DolanMore02}.  Specifically, if $\mathcal{P}$ denotes the set of test problems,
the performance profiles plot the function
$\pi_s:[0,r_M]\rightarrow\Re^+$ defined by
$$\pi_s(\tau)=\frac{1}{|\mathcal{P}|}\left|\left\{p\in \mathcal{P}: \text{log}_2(r_{p,s})\le \tau\right\}\right|,$$
  where $r_{p,s}$ denotes the ratio of the number
  of function evaluations needed to solve problem $p$ with method $s$ with the least number of function
  evaluations needed to solve $p$ by any method.  Note that $r_M$ is the maximum value  of  $\text{log}_2(r_{p,s})$.

\medskip
\noindent
    {\bf Experiment 1. } 
    The first experiment
compares the number of function evaluations needed to solve each problem using
five different options for initializations (\ref{eqn-B0}) found in Table~\ref{table1}. Three memory size limits ($m=3,5,$ and $7$) were tested.  

\begin{table}[h!]
\caption{Values for  different $\zeta_k$  and $\zeta_k^C$ used  in Experiment 1.}
\centering
 {\renewcommand{\arraystretch}{2}
    \begin{tabular}{c | c | c}
    & $\zeta_k$ & $\zeta_k^C$\\
    \hline
    Option 1 &  $\frac{y_k^Ty_k}{y_k^Ts_k}$ & $\frac{y_k^Ty_k}{y_k^Ts_k}$\\\hline
Option 2 & $ \frac{\text{trace}(Y_k^TY_k)}{\text{trace}(S_k^TY_k)}$ &  $\frac{\text{trace}(Y_k^TY_k)}{\text{trace}(S_k^TY_k)}$\\ \hline
    Option 3 & $\frac{\text{trace}(S_k^TY_k)}{\text{trace}(S_k^TS_k)}$ &  $\frac{\text{trace}(S_k^TY_k)}{\text{trace}(S_k^TS_k)}$ \\ \hline
    Option 4 &  $\underset{k-1\le i \le k-l}{\text{max}}\left\{\frac{y_i^Ty_i}{y_i^Ts_i}\right\}$ & $\frac{y_k^Ty_k}{y_k^Ts_k}$\\ \hline
 Option 5 &  $\underset{k-1\le i \le k-l}{\text{max}}\left\{\frac{y_i^Ty_i}{y_i^Ts_i}\right\}$ & $\underset{k-1\le i \le k-l}{\text{mean}}\left\{\frac{y_i^Ty_i}{y_i^Ts_i}\right\}$\\ \hline
\end{tabular} } 
\label{table1}   
\end{table}

\medskip

The first initialization option tested is the Barzilai and Borwein ({\small BB}) initialization, which is
a standard choice for many quasi-Newton methods~\cite{bb88}.  This is
a one-parameter initialization that does not use a dense
initialization; however, we found that this is a very
competitive initialization for the {\small MSS} method.  The second and third
options are  one-parameter initialization that pick $\zeta_k$ based on Theorem~\ref{theoremjan2888-2}.  
The fourth option is
a two-parameter initialization that $\zeta_k$ to be the
maximum ratio of $y_i^Ty_i/y_i^Ts_i$ where $i$ ranges over the stored
updates, and $\zeta_k^C$ is chosen as the {\small BB} initialization.  This initialization
is based on ideas in~\cite{Oleg2021}.  Finally, the last option can be viewed as a variant of Option 4.
 In all experiments, safeguarding was
use to ensure that $\zeta_k$ and $\zeta_k^C$ were neither too large nor
too small.  Specifically, for these experiments,
either parameter was set to its previous value
if it fell outside the range $[10^{-4},10^4]$.
In addition to preventing a parameter from becoming too large, this also prevents
either parameter becoming negative leading to indefinite trust-region subproblems
(see the discussion at the end of Section~\ref{subsec-dense}).
Finally, for the initial iteration,  $\zeta_0=\zeta_0^C=1$.
The parameters chosen in Table~\ref{table1} represent some of the best  choices that
we have found.

\medskip
  
  For all initializations, $m=3$ yielded the best results for {\small MSS}.  Morever, 
  Options 1, 4, and 5 were the best choices with $m=3$.  
  Figure~\ref{fig1} displays these results.  From the performance profile, it appears that Option 4 outperforms the other two options, including usual quasi-Newton initialization (Option 1).  
It is important to note that the {\small MSS} method performs the same
number of function as gradient evaluations (i.e., one per iteration),
and thus, the performance profile for gradient evaluations is
identical to the one presented for function evaluations.

\begin{figure*}[h!]	
		\includegraphics[width=\textwidth]{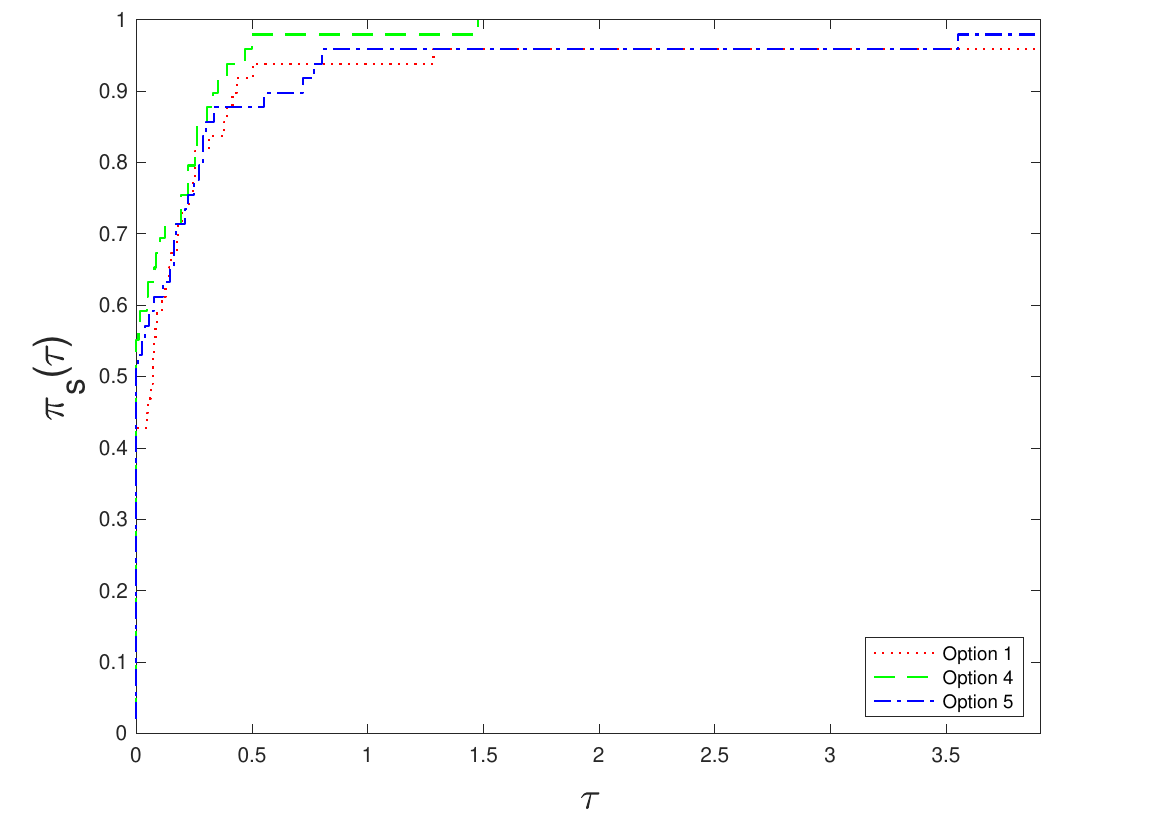}
		\caption{Performance profile comparing function evaluations using MSS with three different initializations and $m=3$.}
		\label{fig1}       
\end{figure*}

  \medskip
  
  Table~\ref{table1} presents the cumulative results for $m=3$ for all five initializations.  In this table, the total number of function
  evaluations reported includes only the 46 problems on which the {\small MSS} method converged
  using all five initializations.  Table~\ref{table1} shows that Option 4 appears to outperform the other
  initializations in number of problems {\small MSS} could solve but also in terms of the total
  number of function evaluations when considering only the problems in which all initializations led to
  convergence for {\small MSS}.

\begin{table}[h!]
  \caption{Cumulative results on 49 CUTEst problems for $m=3$.}
  \label{table1}
  \begin{center}
    \begin{tabular}{ l c c c c c }
       & Option 1 & Option 2 & Option 3 & Option 4 & Option 5\\
       \hline
   Problems solved & 47 & 48 & 48 & 49 & 48\\ 
   Function evaluations (FE)  &  8295 & 7570 &9255&  6809 & 6944 \\
\hline    \end{tabular}
  \end{center}
  \end{table}

  Finally, for comparison between the different memory allocations, we report in Table~\ref{table2}
  the performance of Option 4 using all three memory size $m=3, 5,$ and $7$.  The number of
  function evaluations found in Table~\ref{table2} for $m=3$ are different than in Table~\ref{table1} since Table~\ref{table1} reports results only on 46 problems; in contrast, function evaluations for all 49 problems are included in Table~\ref{table2} since using {\small MSS} converged on all 49 problems for all three memory sizes using Option 4.
    
\begin{table}[h!]
  \caption{Cumulative results on 49 CUTEst problems using MSS with Option 4.}
  \label{table2}
  \begin{center}
    \begin{tabular}{ l c c c c c }
       memory & $m=3$ & $m=5$ & $m=7$ \\
       \hline
   Problems solved & 49 & 49 &  49 \\
   Function evaluations (FE)  &  12,457 & 14,613  & 15,232\\
\hline    \end{tabular}
  \end{center}
  \end{table}
  

\noindent
  {\bf Experiment 2.}
One of the most-used indefinite updates for large-scale
optimization is the limited-memory symmetric rank-one ({\small L-SR1}). 
 For  this experiment,
we implemented  a  standard  {\small L-SR1} trust-region method~\cite[Algorithm 6.2]{nocedal2006numerical} and chose the approximate
solution to the subproblem using a truncated-CG (see, for example,~\cite{nocedal2006numerical,conn2000trust}).
The {\small LSR-1} method was initialized using $B_0=\gamma^{SR1}_{k} I$, where $\gamma^{SR1}_k$
could be changed each iteration.  
Updates  to the  quasi-Newton pairs were accepted provided
\begin{equation}\label{sr1-require}|s_k^T(y_k-B_ks_k)|\ge \tau_s\|y_k-B_ks_k\|\|s_k\|,
 \end{equation}
where $\tau_s=1e^{-8}$.  For these experiments, the initial iteration
used $B_0=I$, and for subsequent iterations, the {\small LSR-1} method
was initialized two standard ways: (1) using the {\small BB} initialization and
(2) using  $B_0=I$ for every initialization.  
Since we require
(\ref{sr1-require}) to update the pairs, there was no danger of the
denominator for the {\small BB} initialization becoming zero.  For these
results, the trust-region framework (e.g., parameters) for the {\small MSS} and
{\small L-SR1} methods were identical. The results of this
experiment are presented in
a performance profile and table.

\begin{figure*}[h!]	
		\includegraphics[width=\textwidth]{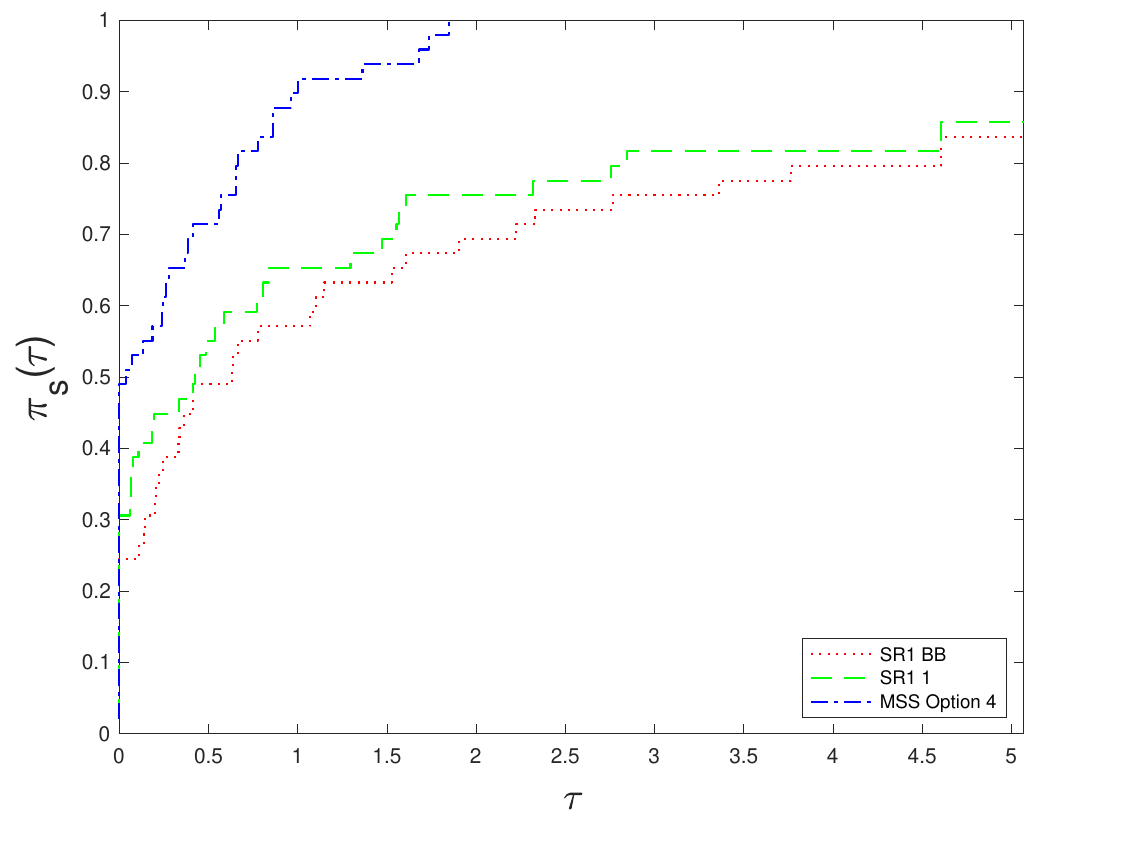}
		\caption{Performance profile comparing function evaluations using L-SR1 and MSS with  $m=3$. "SR1 BB" refers to using L-SR1 together with the BB initialization and "SR1 1" refers to using L-SR1 together with the initialization $B_0=I$ each iteration.}
		\label{sr1}       
\end{figure*}

Figure~\ref{sr1}
contains the performance profile for function evaluations.  From the figure, it
appears that {\small MSS} with Option 4 outperforms both implementations
of {\small L-SR1} in terms of function evaluations using $m=3$.
Table~\ref{table3} contains cumulative results for $m=3$ for {\small MSS} and the two L-SR1 methods. 
Of the 49 problems that at least one method converged on, {\small MSS} converged on all 49; meanwhile,
the {\small L-SR1} methods converged on 41 and 42 depending on the initialization.  There were a total
of 40 problems in which all three method converged; the total number of function evaluations on these
40 problems are reported in the last line of table.  On the problems in which all three methods converged,
{\small MSS} required significantly fewer function evaluations.  In table~\ref{table4}, the number of problems
solved and the number of function evaluations are reported for memory sizes $m=3$, $m=5$ and $m=7$.  In
all cases, {\small MSS} outperformed {\small L-SR1}.  Note that the total number of function evaluations
only counts the 38 problems on which all methods converged using all three memory sizes.

\begin{table}[h!]
  \caption{Cumulative results on 49 CUTEst problems using MSS with Option 4 and L-SR1.}
  \label{table3}
  \begin{center}
    \begin{tabular}{ l c c c c c }
        & MSS & L-SR1 BB &  L-SR1 1 \\
       \hline
   Problems solved & 49 & 41 &  42 \\
   Function evaluations (FE)  &  3,120 & 14,621  & 11,817\\
\hline    \end{tabular}
  \end{center}
  \end{table}

\begin{table}[h!]
  \caption{Cumulative results on 49 CUTEst problems using MSS with Option 4 and L-SR1 for various memory sizes.}
  \label{table4}
  \begin{center}
    \begin{tabular}{ l l c c }
Memory & Solver & Problems solved  & Function evaluations (FE) \\
 & MSS Option 4& 49 &  2,914 \\
$m=3$ & L-SR1 BB & 41 & 14,142\\
 & L-SR1 1 & 42 & 11,643 \\
 \hline
 & MSS Option 4 & 49 &3,069\\
 $m=5$ & L-SR1 BB & 40 & 12,823\\
 & L-SR1 1& 44 & 12,928\\
 \hline
& MSS Option 4 & 49 & 3,005\\
 $m=7$ &L-SR1 BB & 41 & 13,200\\
 & L-SR1 1 & 43 & 11,660\\
\hline    \end{tabular}
  \end{center}
  \end{table}

Figure~\ref{sr1_time} contains the performance profile for total time needed to solve each problem
for the case $m=3$.  From the profile, {\small MSS} with Option 4 takes significantly less time than
 either {\small L-SR1} method.  For the cases $m=5$ and $m=7$, the results are similar.

\begin{figure*}[h!]	
		\includegraphics[width=\textwidth]{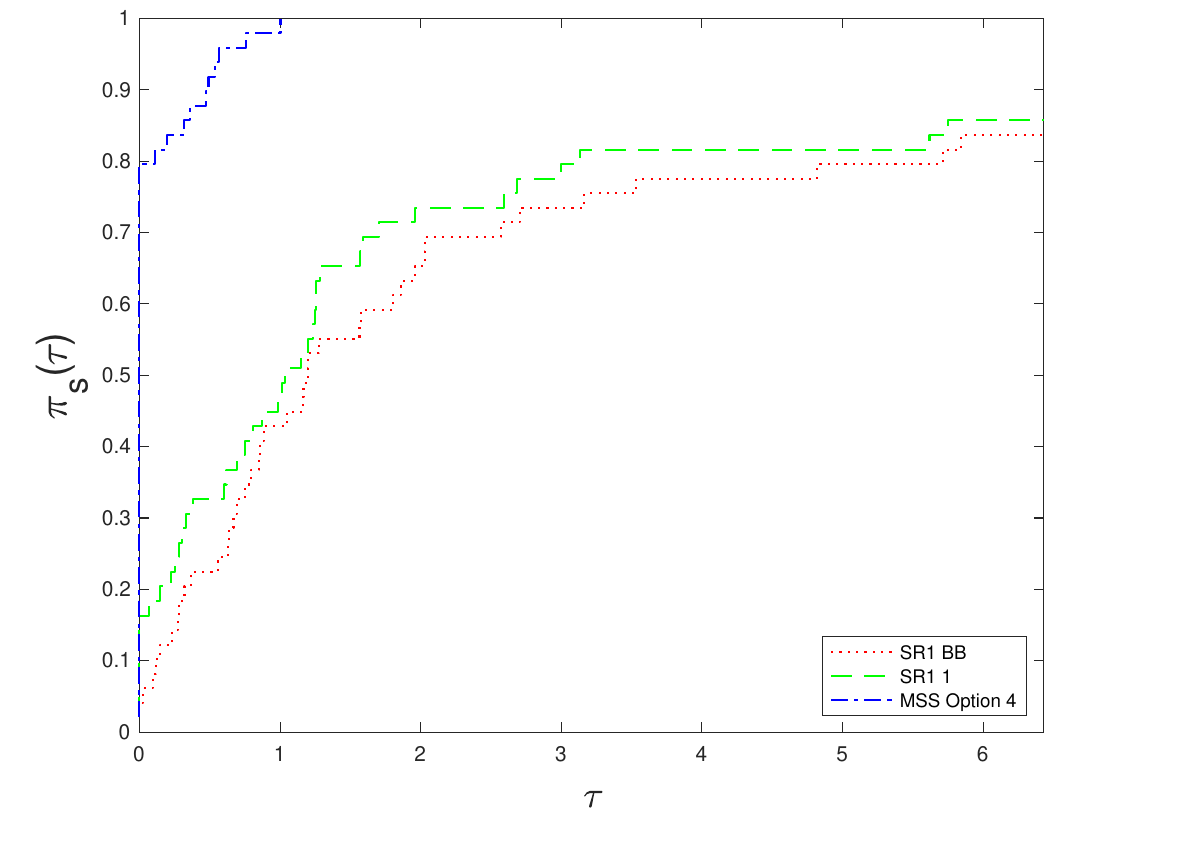}
		\caption{Performance profile comparing runtime using L-SR1 and MSS with  $m=3$. "SR1 BB" refers to using L-SR1 together with the BB initialization and "SR1 1" refers to using L-SR1 together with the initialization $B_0=I$ each iteration.}
		\label{sr1_time}       
\end{figure*}

\medskip
 
\noindent
{\bf Experiment 3.}
In this section, we compare the {\small MSS} method with a dense initialization to the limited-memory Powell-symmetric-Brodyen ({\small PSB}) method.  Similar to the {\small MSS} update, the {\small PSB} update is a rank-2 quasi-Newton update that can generate indefinite Hessian approximations; thus, it is natural to use both methods in a trust-region framework.  Moreover, {\small PSB} admits a compact representation~\cite{powell_psb} enabling us to use the same {\small{OBS}}-like solver that used for {\small MSS}.  In other words, to use {\small PSB}, the only
 changes in Algorithm 1 for {\small MSS} that need to be made to use {\small PSB} are (i) the definition of $M$, (ii) the definition
 of $\Psi$, and (iii) the removal of the requirement that $S$ is full rank in line 18.  (For more details on the {\small PSB} update,
 see, e.g.,~\cite{powell_psb}.)  For these experiments, a limited-memory {\small PSB} method was initialized using (i) $B_0=\gamma_kI$,
 where $\gamma_k=(y_k^Ty_k)/(s_k^Ty_k)$ (i..e, the {\small BB} initialization) and (ii) 
 $B_0=\gamma_k I$,
 where $\gamma_k=(s_k^Ty_k)/(s_k^Ts_k)$, which is denoted as the "o{\small BB}" initialization in the table.  Note that both these initialization are inspired by the results of 
(Theorem~\ref{theoremjan2888-2}) with the stored vectors limited to the most current quasi-Newton pair (i.e., the 
number of stored pairs is one).
 
 Figure~\ref{psb1}--\ref{psb3}
contains the performance profile for function evaluations.  From the figure, it
appears that {\small MSS} with Option 4 outperforms both implementations
of {\small L-PSB} in terms of function evaluations using $m=3$, $m=5$, and $m=7$, respectively.  
Table~\ref{tablePSB} lists cumulative results for these different memory sizes and initializations.
To compute the total number of function evaluations in the table, only the problems on which all three methods
using all three memory sizes were considered--this amounted to 45 problems.  From the table, {\small MSSM}
performs best on this test set for smaller memory sizes; in contrast, {\small L-PSB} with the {\small BB} initialization
solved more problems with smaller memory sizes but the total number of function evaluations on the 45 problems
was smallest with $m=5$.  Meanwhile, {\small L-PSB} with the o{\small BB} initialization performed better with
larger memory sizes.

\begin{figure*}[h!]	
		\includegraphics[width=\textwidth]{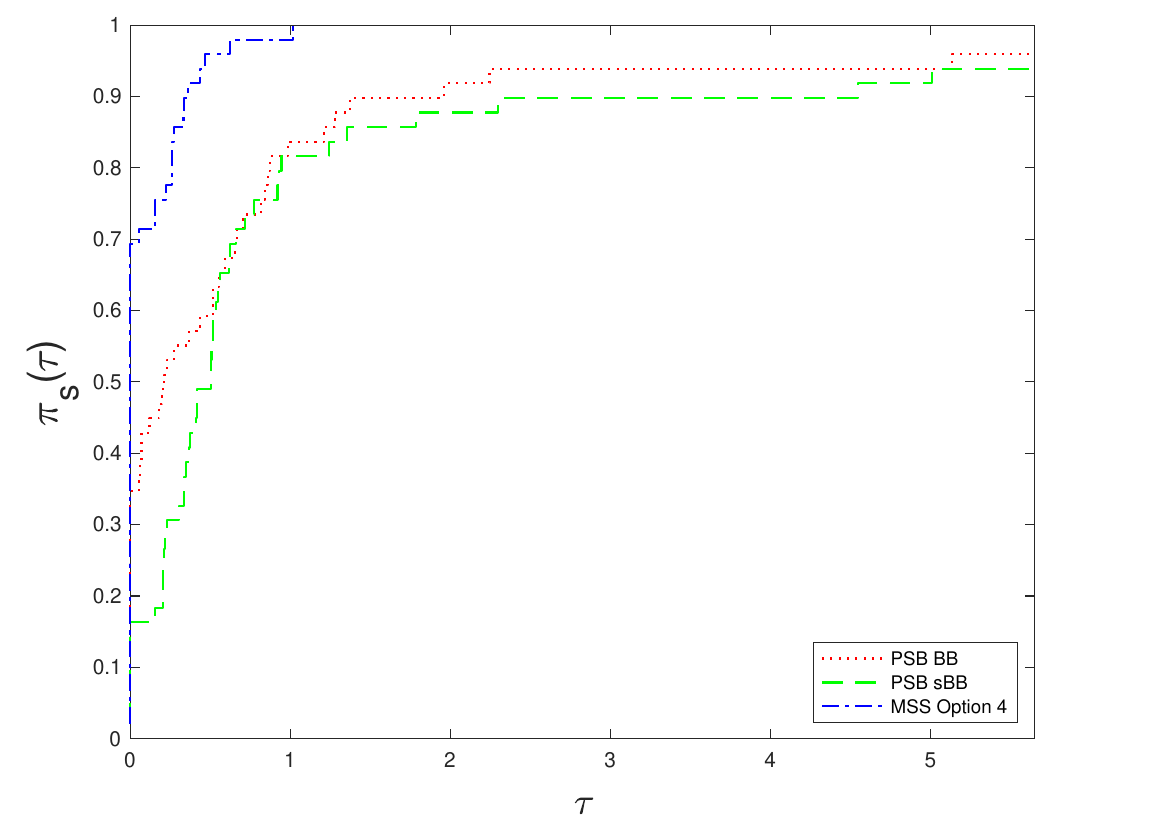}
		\caption{Performance profile comparing function evaluations using L-PSB and MSS with  $m=3$. "PSB BB" refers to using L-PSB together with the BB initialization and "PSB sBB" refers to using L-PSB together with the initialization $B_0=(s_k^Ty_k)/(s_k^Ts_k)I$ each iteration.}
		\label{psb1}       
\end{figure*}

\begin{figure*}[h!]	
		\includegraphics[width=\textwidth]{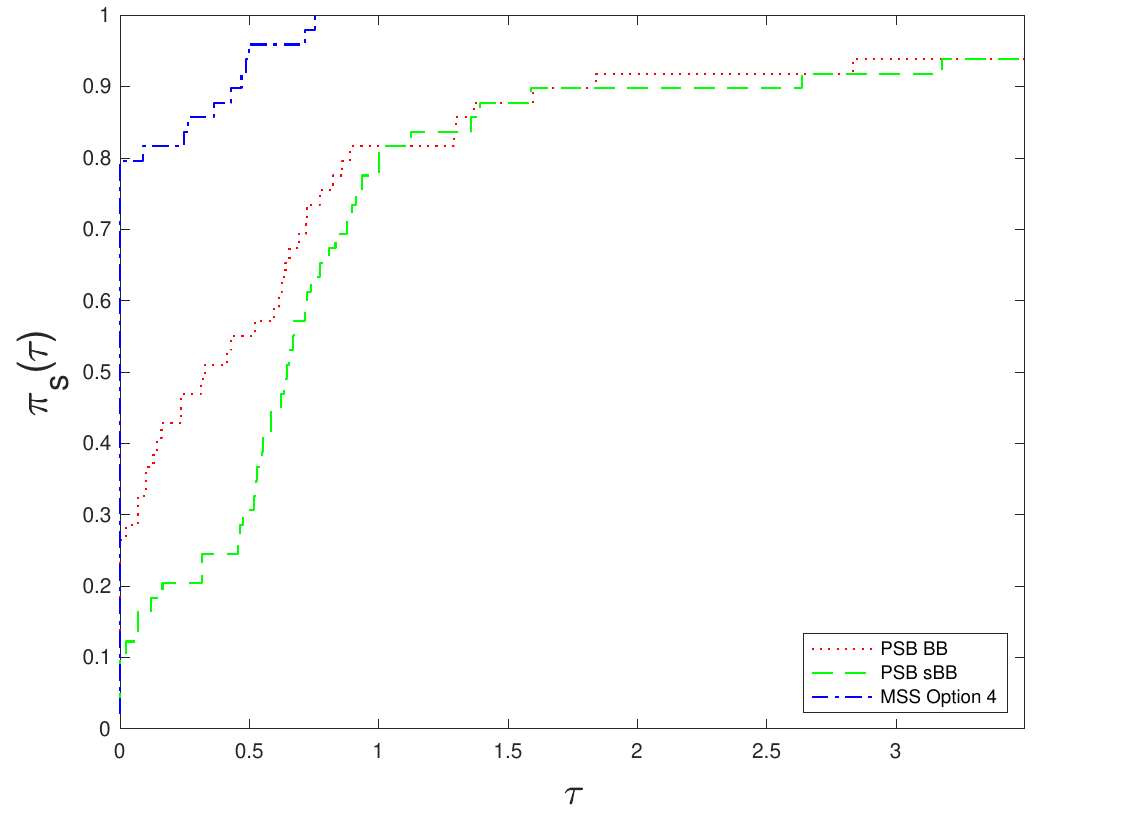}
		\caption{Performance profile comparing function evaluations using L-PSB and MSS with  $m=5$. "PSB BB" refers to using L-PSB together with the BB initialization and "PSB sBB" refers to using L-PSB together with the initialization $B_0=(s_k^Ty_k)/(s_k^Ts_k)I$ each iteration.}
		\label{psb2}       
\end{figure*}

\begin{figure*}[h!]	
		\includegraphics[width=\textwidth]{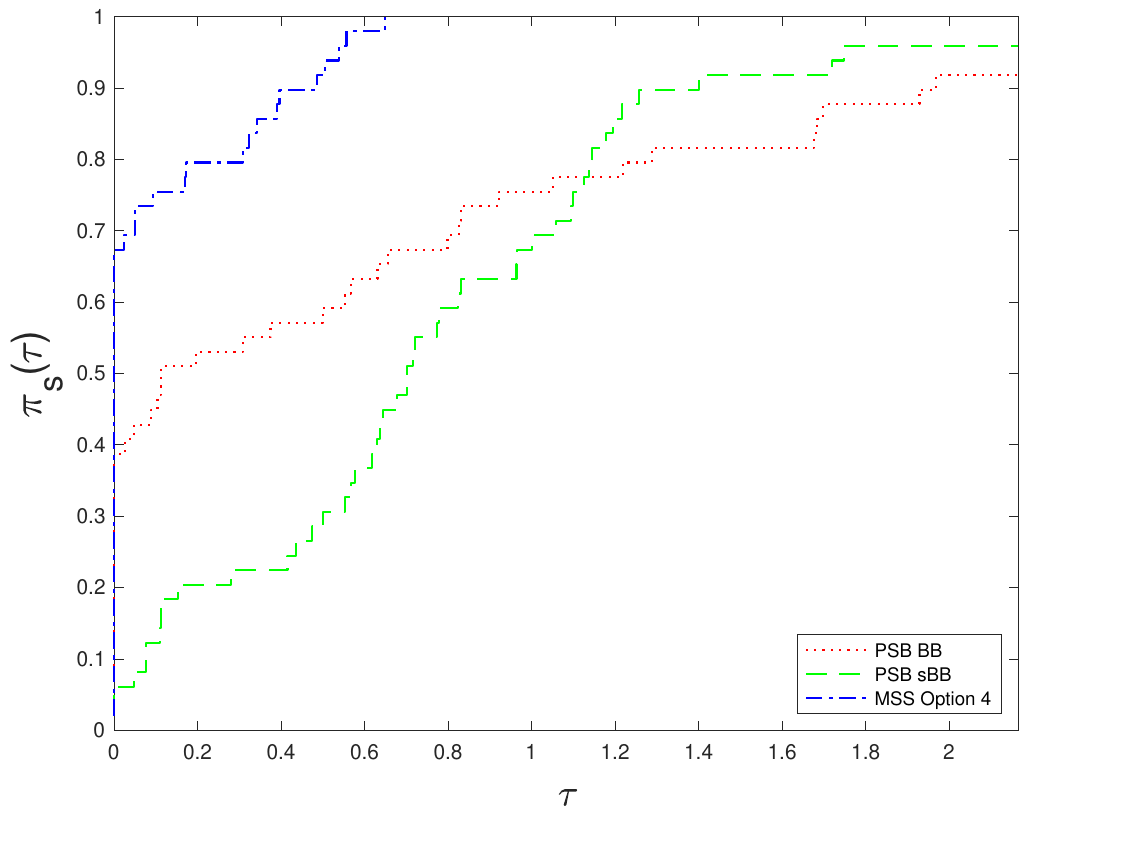}
		\caption{Performance profile comparing function evaluations using L-PSB and MSS with  $m=7$. Here, "PSB BB" refers to using L-PSB together with the BB initialization and "PSB sBB" refers to using L-PSB together with the initialization $B_0=(s_k^Ty_k)/(s_k^Ts_k)I$.}
		\label{psb3}       
\end{figure*}

\begin{table}[h!]
  \caption{Cumulative results on 49 CUTEst problems using MSS with Option 4 and L-PSB for various memory sizes.  Here, "PSB BB" refers to using L-PSB together with the BB initialization and "PSB sBB" refers to using L-PSB together with the initialization $B_0=(s_k^Ty_k)/(s_k^Ts_k)I$.}
  \label{tablePSB}
  \begin{center}
      \begin{tabular}{ l l c c }
  Memory & Solver & Problems solved  & Function evaluations (FE) \\
 & MSS Option 4& 49 & 5927  \\
$m=3$ & L-PSB BB & 47 &11,336 \\
 & L-PSB oBB&46  & 14.529 \\
 \hline
 & MSS Option 4 & 49&\ 6,290\\
 $m=5$ & L-PSB BB &46  & 11,157\\
 & L-PSB oBB& 46 & 10,705 \\
 \hline
& MSS Option 4 & 49 & 6,648 \\
 $m=7$ &L-PSB BB & 45& 12,951\\
 & L-PSB oBB & 47 & 10,694\\
  \hline    \end{tabular}
  \end{center}
  \end{table}

Finally, the performance profile for total runtime on the test set of 49 problems with $m=3$ for {\small MSS} with Option 4
and the two {\small PSB} methods are found in Figure~\ref{psb_time3}.   This figure shows that {\small MSS} with Option 4
was the faster of the three methods on this test set.

\begin{figure*}[h!]	
		\includegraphics[width=\textwidth]{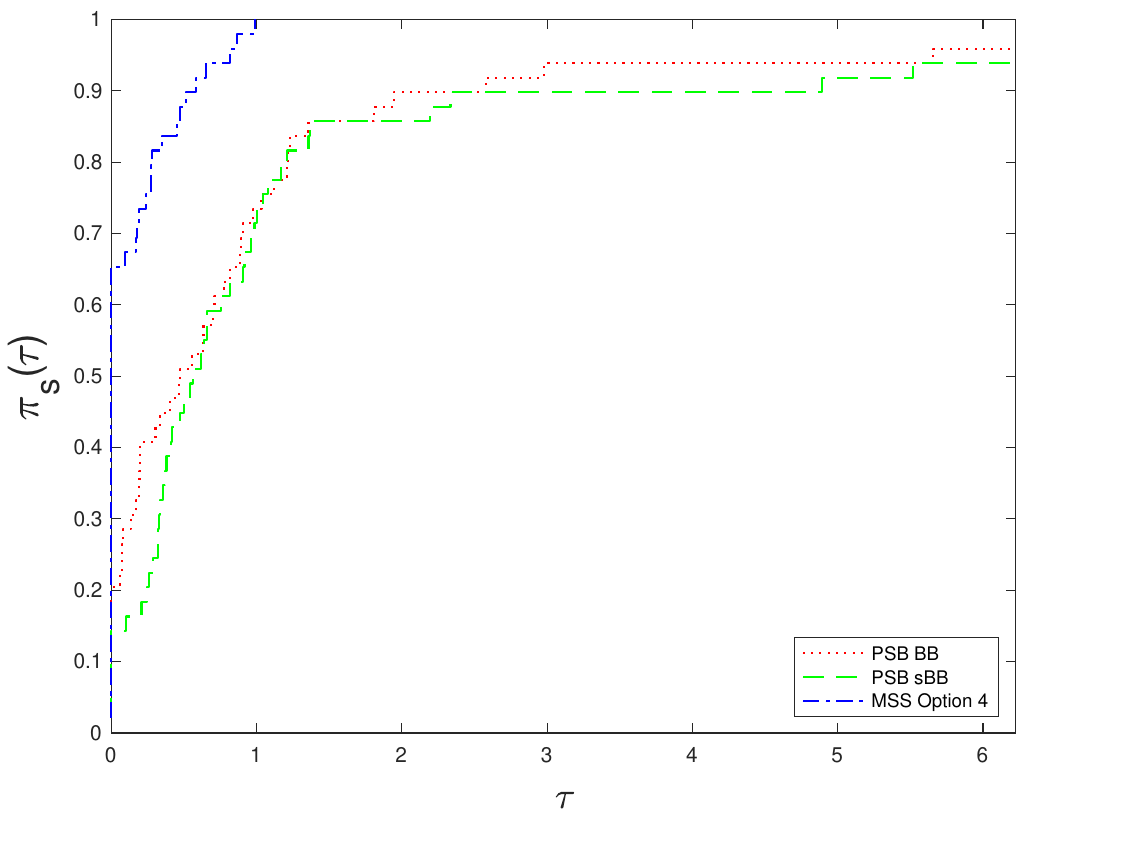}
		\caption{Performance profile runtime using L-PSB and MSS with  $m=3$. Here, "PSB BB" refers to using L-PSB together with the BB initialization and "PSB sBB" refers to using L-PSB together with the initialization $B_0=(s_k^Ty_k)/(s_k^Ts_k)I$.}
		\label{psb_time3}       
\end{figure*}

%% file: conclusion.tex
\section{Concluding Remarks}
In this paper, we demonstrated how a dense initialization can be used
with a {\small MSS} method.  In numerical experiments on {\small CUTE}st
test problems, the dense initialization performs better than standard constant
diagonal initializations.  Results suggest that this method performs
well with small memory sizes (i.e., $m=3$ outperformed the other memory
sizes and the number of function evaluations generally increased as
$m$ increased).
Moreover, with small $m$, the proposed method generally
outperforms a basic {\small L-SR1} trust-region method.  Further
research includes considering other symmetrization, choices for the
two parameters $\zeta$ and $\zeta^C$, and trust-region norms
(including the so-called \emph{shape-changing norm} first proposed
in~\cite{BYuan02}).  The results of future research will be
submitted in subsequent papers.

%% file: acknowledgements.tex
\subsection{Acknowledgement.}
The authors would like to thank the referees. In particular, one referee provided
an explanation that led to the inclusion of Lemma~\ref{lemma-Y} and significantly improved
the entire section.  In fact, part of the proof is based on an observation
by the referee.